\documentclass[reqno,11pt]{amsart}
\usepackage{latexsym,amssymb,amsfonts,amsmath,amsthm,mathtools,bbm}
\usepackage[usenames]{xcolor}
\usepackage{eucal}
\usepackage[colorlinks,citecolor=blue,urlcolor=blue]{hyperref}
\usepackage{graphicx}
\newtheorem{thm}{Theorem}[section]

\newtheorem{lem}[thm]{Lemma}

\newtheorem{defn}[thm]{Definition}

\theoremstyle{remark}

\numberwithin{equation}{section}
\newcommand{\Real}{\mathbb R}
\newcommand{\eps}{\varepsilon}

\newcommand{\F}{\mathcal{F}}

\newcommand{\one}[1]{\mathbf{1}_{\{#1\}}}

\renewcommand{\P}{\mathbb{P}}

\newcommand{\E}{\mathbb{E}}

\allowdisplaybreaks

\begin{document}

\title[Multitype PCR branching processes]{Multitype PCR branching processes}

\author{P. Chigansky}
\address{Department of Statistics and Data Science,
The Hebrew University of Jerusalem, Mount Scopus, 
Israel}
\email{pavel.chigansky@mail.huji.ac.il}

\author{F. Klebaner}%
\address{School of Mathematics, Monash University, Clayton, Vic. 3058, Australia}
\email{fima.klebaner@monash.edu} 

\author{M. Mrksa}
\address{School of Mathematics, Monash University, Clayton, Vic. 3058, Australia}
\email{marko.mrksa@monash.edu} 

\author{S. Sagitov}
\address{Mathematical Sciences, Chalmers University of Technology, Gothenburg, Sweden}
\email{serik@chalmers.se} 

\date{\today}
\begin{abstract}
To model amplification Polymerase Chain Reaction (PCR) techniques targeting DNA sequences of several types, we introduce a multitype PCR branching process as a generalized version of the Michaelis--Menten–based branching process model introduced in \cite{JagKlePCR}. 
We establish two limit theorems extending the results of \cite{Fima17PCR} to the multitype case.
\end{abstract}
\maketitle

\section{Introduction}

Polymerase Chain Reaction (PCR) is a laboratory technique used to amplify specific DNA sequences, producing millions of copies from a small initial sample. In practice, it is widely used in forensic science, cancer diagnostics, detection of infectious diseases, genetic testing, evolutionary biology, and molecular cloning, among many other applications in biomedical research.

A mathematical PCR amplification model introduced in \cite{JagKlePCR} describes the number of target DNA sequences $Z(n)$ after $n$ PCR cycles as a population size--dependent branching process. Given $Z(n-1)=z$, each of the $z$ sequences independently replicates successfully with probability
\[
p(v,K,z)=\frac{vK}{K+z},
\]
where $v\in(0,1]$ and $K>0$, and fails to replicate otherwise. If $z'$ denotes the number of successful replications, then
\[
Z(n)=z+z'.
\]

We refer to this model as a single-type PCR branching process with parameters $(K,v)$. Here $K$ is the so-called Michaelis--Menten constant. Its exact value, typically assumed to be large, is difficult to specify, as it depends on experimental conditions and PCR reaction chemistry. A natural object for mathematical study, the density process $X(n)=K^{-1}Z(n)$, was analyzed in \cite{Fima17PCR}.

The boundary case $K=\infty$, for which $p(v,K,z)=v$, reduces the model to a Galton--Watson process $Y(n)$ with mean offspring number $b=1+v$. By the classical theory of Galton--Watson processes, the martingale $b^{-n}Y(n)$ has an almost sure limit $W$ with
\[
\E W = Y(0),
\]
see \cite{AN,Stochalc}. The Galton--Watson process $Y(n)$ with $Y(0)=Z(0)$ describes the early trajectory of $Z(n)$. As shown in \cite{Fima17PCR}, an informative asymptotic formula for $X(\kappa)$ is available for the specific time $\kappa$ at which the two normalization factors coincide,
\[
b^\kappa=K.
\]
Namely (in this paper all asymptotic relations, unless explicitly stated otherwise, assume that $K\to\infty$),
\[
X(\kappa)\xrightarrow{\P} H(W),
\]
where $H$ is a certain function uniquely determined by the parameter $v$.

In modern cancer diagnostics \cite{St}, PCR analysis is often performed simultaneously on DNA sequences of several types, which leads naturally to the concept of multitype PCR branching processes described in Section~\ref{BP}. Section~\ref{MR} presents two limit theorems for such processes at times around the pivotal number $\kappa$, in the spirit of \cite{Fima17PCR}.

Since the pivot time $\kappa$ is a logarithmic function of $K$, it is important to verify the theoretical asymptotic formulas using computer simulations. In Section~\ref{Ill} we present two illustrations of our limit theorems based on such simulations. The main illustration addresses the setting of \cite{LS}, demonstrating that our asymptotic results may partly accelerate the ABC procedure used there. 
Section~\ref{Prf} contains the proofs.

\section{Multitype PCR branching process}\label{BP}
In this paper, $\bar{x}$ denotes the vector $(x_1,\ldots,x_d) \in \mathbb{R}^d$,
regardless of the underlying symbol $x$. 
We write $x = x_1 + \cdots + x_d$ for the sum of all components of $\bar{x}$ and $x_0 = x_1 + \cdots + x_{d_0}$ for the sum of the first $d_0$ components of $\bar{x}$.
In particular, 
\[\bar{Z}(n) = (Z_1(n),\ldots,Z_d(n)),\quad Z(n) = Z_1(n) + \cdots + Z_d(n),\quad Z_0(n) = Z_1(n) + \cdots + Z_{d_0}(n).\]

\begin{defn}\label{def}
 For given $K>0$, vector $\bar v$ such that for some $1\le d_0\le d$, 
\begin{equation}\label{viq}
1\ge v_1=\ldots=v_{d_0}>v_{d_0+1}\ge \ldots\ge v_d>0,
\end{equation}
and initial state
 $\bar  Z(0)\in\mathbb N^d$,
 define a Markov chain  
 $\bar  Z(0), \bar  Z(1), \bar  Z(2), \ldots$,
 by the recursion
\begin{equation}\label{Zn}
Z_i(n)=Z_i(n-1)+\sum_{j=1}^{Z_i(n-1)}\xi_{j}\big(v_i,K,Z(n-1)\big), \quad n\in \mathbb N,\quad i=1,\ldots, d,
\end{equation}
where the summands, conditioned on $Z(n-1)=z$, are independent Bernoulli random variables with success probabilities
\begin{equation}\label{mx}
\P(\xi_{j}(v_i,K,z)=1)=\frac{v_iK}{K+z}.
\end{equation}
Such a Markov chain $\bar  Z(n)$ will be called a multitype PCR (MPCR) branching process with parameters $(K,\bar v)$. 
\end{defn}

Our interest lies in approximating the density process $\bar  X(n)=\frac{1}{K}\bar  Z(n)$ for large values of $K$. By dividing \eqref{Zn} by $K$, a recursive equation for the density is obtained: 
\begin{equation}\label{stochdyn}
\bar  X(n) = \bar  F(\bar  X(n-1)) + \frac 1 {\sqrt K} \bar  \eps(n), \quad n\in \mathbb N, 
\end{equation}
where the vector $ \bar  F(\bar x)$ has components
\begin{equation}\label{barF}
F_i(\bar  x)=x_i+\frac{v_ix_i}{1+x}, \quad i=1,\ldots,d.
\end{equation}
 
The components of the random vectors  $\bar  \eps(n)$ are given by 
\begin{equation}\label{epsn}
\eps_i(n) = \frac{1}{\sqrt K}\sum_{j=1}^{Z_i(n-1)} \Big(\xi_{j}(b_i,K, Z(n-1)) -\frac{v_iK}{K+Z(n-1)}\Big).
\end{equation}
Since by Definition \ref{def}, the summands in \eqref{epsn} have zero means, we have 
 $\eps_i(n) =O_{\P}(1)$ for  any fixed $n$. Thus, if 
$
\frac{1}{K}\bar  Z(0)\xrightarrow{\P} \bar  x,
$
representation \eqref{stochdyn} yields the well known limit for density dependent branching processes, see e.g. \cite{Klethr},
\begin{equation}\label{lln}
\bar   X(n)\xrightarrow{\P} \bar  F^{(n)}(\bar  x),\quad \forall n\in \mathbb N,
\end{equation}
where $\bar  F^{(n)}$ is the $n$-th iterate of $\bar  F$. 

The limit in \eqref{lln}  provides a useful approximation unless $\bar x = \bar 0$. If, however, $\bar Z(0)$ is fixed, so that $\bar x = \bar 0$, it ceases to be informative because $\bar F^{(n)}(\bar 0) = \bar 0$ for any fixed $n$ (since zero is a fixed point of $\bar F$). Nevertheless, a non-trivial limit can still be obtained if the trajectory is observed at times that grow with $K$. The suitable time scale turns out to be logarithmic in $K$, and in this case, unlike in \eqref{lln}, the corresponding limit is random.

\section{Main results}\label{MR}
To formulate our main results concerning the MPCR branching process with parameters $(K,\bar v)$, Theorems \ref{mainthm} and  \ref{thm} below, we introduce the following objects. 

\medskip

\begin{enumerate}

\item  For the MPCR branching process $\bar Z(n)$, controlled by $K$, we consider a coupled vector $\bar Y(n)$ of $d$ independent Galton--Watson processes, constructed in the following way.
Let $U_{i,j,n}\sim U([0,1])$ be i.i.d. uniform random variables, define coupled Bernoulli random values
\begin{equation}\label{xieta}
\xi_{j} (v_i,K,Z(n-1))=\mathbf{1}\left\{U_{i,j,n} \leq  \frac{v_iK}{ K+Z(n-1)}\right\},\quad    
\eta_{i,j,n} = \mathbf{1}\big\{U_{i,j,n} \leq v_i\big\},
\end{equation}
and set recursively 
\begin{equation}
\label{Yn}
Y_i(n)=Y_i(n-1) + \sum_{j=1}^{Y_i(n-1)} \eta_{i,j,n}.
\end{equation}
Since $\frac{v_iK}{ K+Z(n-1)}\le v_i$, the coupled processes defined by \eqref{Zn} and \eqref{Yn} satisfy 
\begin{equation}\label{ZleY}
Z_i(n)\le Y_i(n), \quad  i=1,\ldots,d,\quad n\in\mathbb N.
\end{equation}
Furthermore, the limits
\begin{equation}\label{branchlim}
W_i := \lim_{n\to\infty} b_i^{-n} Y_i(n), \quad b_i = 1 + v_i, \quad i = 1,\ldots,d,
\end{equation}
exist both almost surely and in $\mathbb{L}_2$, and satisfy
\[
\E W_i = Z_i(0).
\]
\item Define a function
\begin{equation}\label{fx}
f(r) = r+\dfrac{v_1r}{1+r},\quad r\ge0,
\end{equation}
and let $f^{(n)}$  for $n\in\mathbb N$, denote its $n$-th iterate. 
Letting $f^{(0)}(r)\equiv r$
and $f^{(-1)}$ be the inverse function of $f$,  define $f^{(-n)}$ for $n\in\mathbb N$, as $n$-th iterate of $f^{(-1)}$.

As shown in \cite[Section~3]{Fima17PCR}, the limit
\begin{equation}\label{Hx}
H(r) := \lim_{n\to\infty} f^{(n)}(b_1^{-n}r), \quad r \ge0,
\end{equation}
exists and defines an increasing continuous function. It is also shown there that
\[
r - r^2 \le f^{(n)}(b_1^{-n}r) \le r, \quad r \ge0,
\]
which implies 
\begin{equation}\label{Hr}
r - r^2  \le H(r) \le r, \quad r \ge0,
\end{equation}
and therefore $H(0)=0$ and $H'(0)=1$.\\

\item 
Define, for $i=1,\ldots,d$, the functions
\begin{equation}\label{Gilim}
G_i(r)=\prod_{m=1}^{\infty}\frac{1+b_i^{-1}H(b_1^{-m}r)}{1+H(b_1^{-m}r)},
\qquad r\ge0.
\end{equation}
In the proof below we show that these infinite products converge and yield positive, continuous, decreasing functions $G_i$ satisfying
\begin{equation}\label{GH}
rG_i(r)=H(r), \qquad i=1,\ldots,d_0,
\end{equation}
see Figure~\ref{fig0}   for an illustration.

\begin{figure}
    \centering
    \includegraphics[height=5cm]{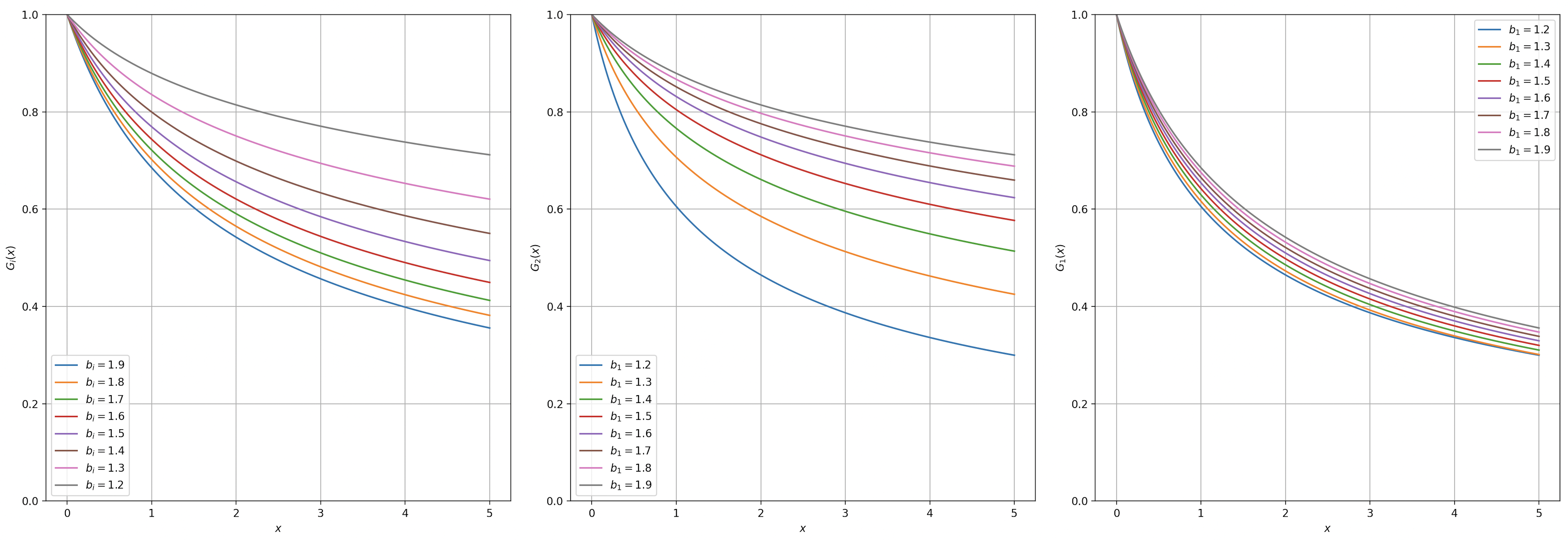}
\caption{\small Graphs of functions defined by \eqref{Gilim}. Left panel: curves $G_1(x),\ldots,G_8(x)$ for $b_1=1.9$, $b_2=1.8$, $b_3=1.7$, $b_4=1.6$, $b_5=1.5$, $b_6=1.4$, $b_7=1.3$, $b_8=1.2$. Middle panel: curves $G_2(x)$ for varying values of $b_1$ with $b_2=1.2$ fixed. Right panel: curves $G_1(x)$ for varying values of $b_1$.}    \label{fig0}
\end{figure}

%

\end{enumerate}

\medskip

\begin{thm}\label{mainthm}
Consider the MPCR branching process $\bar Z(n)$. Given that parameters  \eqref{viq} and the initial vector
$\bar  Z(0)$ are fixed, assume that parameter $K\in\mathbb (0,\infty)$ varies in such a way that 
$$K\to\infty,\quad \log_{b_1}K\in\mathbb N.$$
 Then putting $ \kappa=\log_{b_1}K$, we have
\begin{align}\label{claim2}
(b_1^{-\kappa}Z_1(\kappa), b_1^{-\kappa}Y_1(\kappa),\ldots, &b_d^{-\kappa}Z_d(\kappa), b_d^{-\kappa}Y_d(\kappa))\\
& \xrightarrow{\P}  (W_1G_1(W_0),W_1,\ldots, W_dG_d(W_0),W_d), \nonumber
\end{align}
where $W_1,\ldots,W_d$ are independent random values defined by \eqref{branchlim} and
\(W_0=W_1+\ldots+W_{d_0}.\)
\end{thm}

\medskip

Note that by \eqref{GH},
\[W_iG_i(W_0)=\frac{W_i}{W_0}H(W_0),\quad i=1,\ldots,d_0,\]
yielding the following straightforward corollary of Theorem \ref{mainthm}: 
$$X(\kappa)\xrightarrow{\P}  H(W_0).$$

Theorem \ref{mainthm} is too narrow for application purposes as it addresses just one specific observation time $\kappa=\log_{b_1}K$. The next theorem  widens the range of times $n$ for which the approximation for $\bar Z(n)$ works. 

\begin{thm}\label{thm}
Consider the density process $\bar  X(n)$ of the MPCR branching process. Under the assumptions of Theorem \ref{mainthm}, for any fixed integer $n\in \mathbb Z$, we have
\begin{equation}\label{Xsum}
X(\kappa+n) \xrightarrow{\P}  f^{(n)}(H(W_0)),
\end{equation}
and moreover,
\begin{equation}\label{barX}
\bar  X(\kappa+n) \xrightarrow{\P}  \Big(\frac{W_1}{W_0},\ldots ,\frac{W_{d_0}}{W_0},0,\ldots,0\Big) f^{(n)}(H(W_0)).
\end{equation}
\end{thm}

\section{Illustrations of Theorems \ref{mainthm} and \ref{thm}}\label{Ill}

Figure \ref{fig1} and Figure~\ref{fign} illustrate Theorem \ref{mainthm} in the case $d=2$ with 
$v_1=0.9$ and $v_2=0.2$. We assume $Z_1(0)=Z_2(0)=1$ and $\kappa=25$. 
Note that $\kappa=25$ corresponds to the value of the Michaelis--Menten 
constant $K=(1.9)^{25}\approx 9\ 300\ 000$. 

The left panel of Figure \ref{fig1} compares $b_1^{-\kappa}Z_1(\kappa)$ with $W_1G_1(W_1)=H(W_1)$, 
while the right panel compares $b_2^{-\kappa}Z_2(\kappa)$ with $W_2G_2(W_1)$, 
cf. the left panel Figure \ref{fig0}, which shows the corresponding functions 
$G_1(x)$ and $G_2(x)$ as $G_1(x)$ and $G_8(x)$.
\begin{figure}[htbp]
    \centering
    \includegraphics[height=5.5cm]{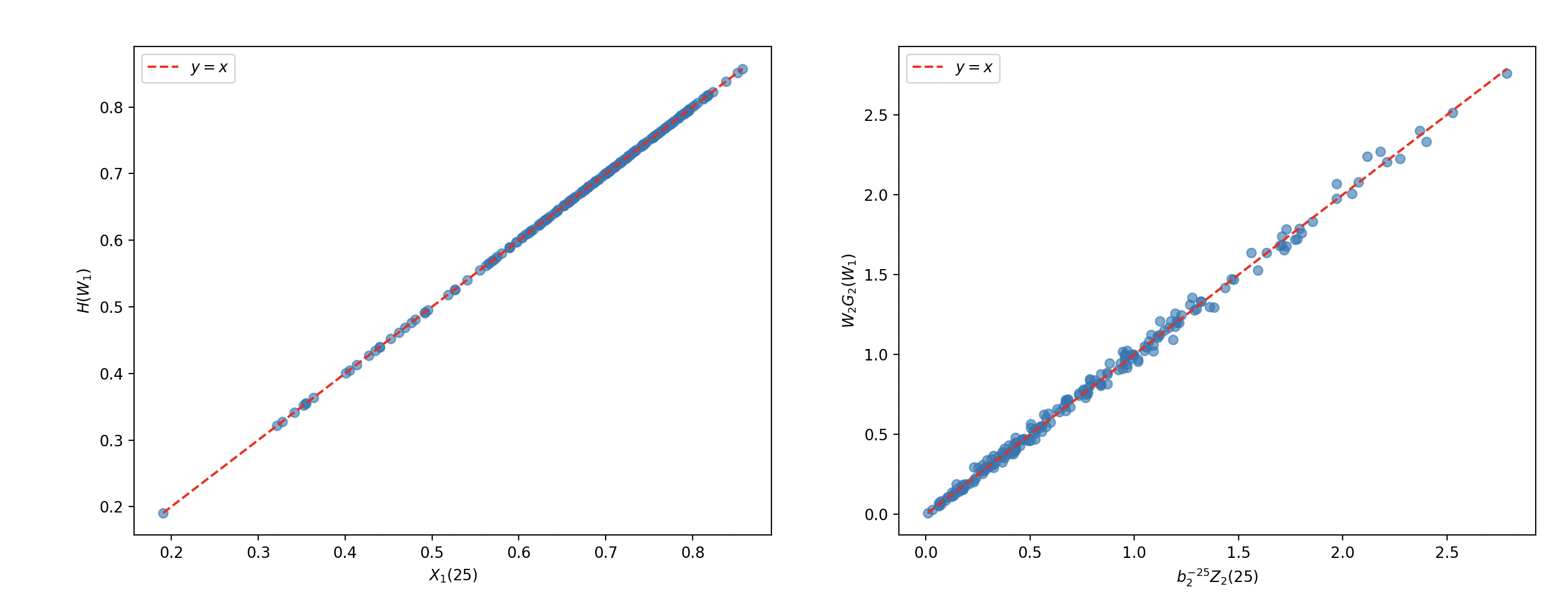}
    \caption{\small An illustration of Theorem \ref{mainthm} in the case $d=2$ with $b_1=1.9$, $b_2=1.2$ and $\kappa=25$ based on 200 simulations. }
    \label{fig1}
\end{figure}

Figure~\ref{fign} depicts the joint distribution of $H(W_1)$ and $W_2 G_2(W_1)$. The slight negative correlation reflects the competition for shared resources between the two molecular species in the PCR experiment.

Recent papers \cite{MMJ24}, \cite{CHM} propose efficient numerical methods for computing the density of $W_i$ when the offspring distribution has finite support. For our purposes, however, we adopted a simpler approach, generating the random variables $W_i$ by simulating the Galton–Watson process for a sufficiently large number of generations.

\begin{figure}[htbp]
    \centering
 \includegraphics[height=5.5cm]{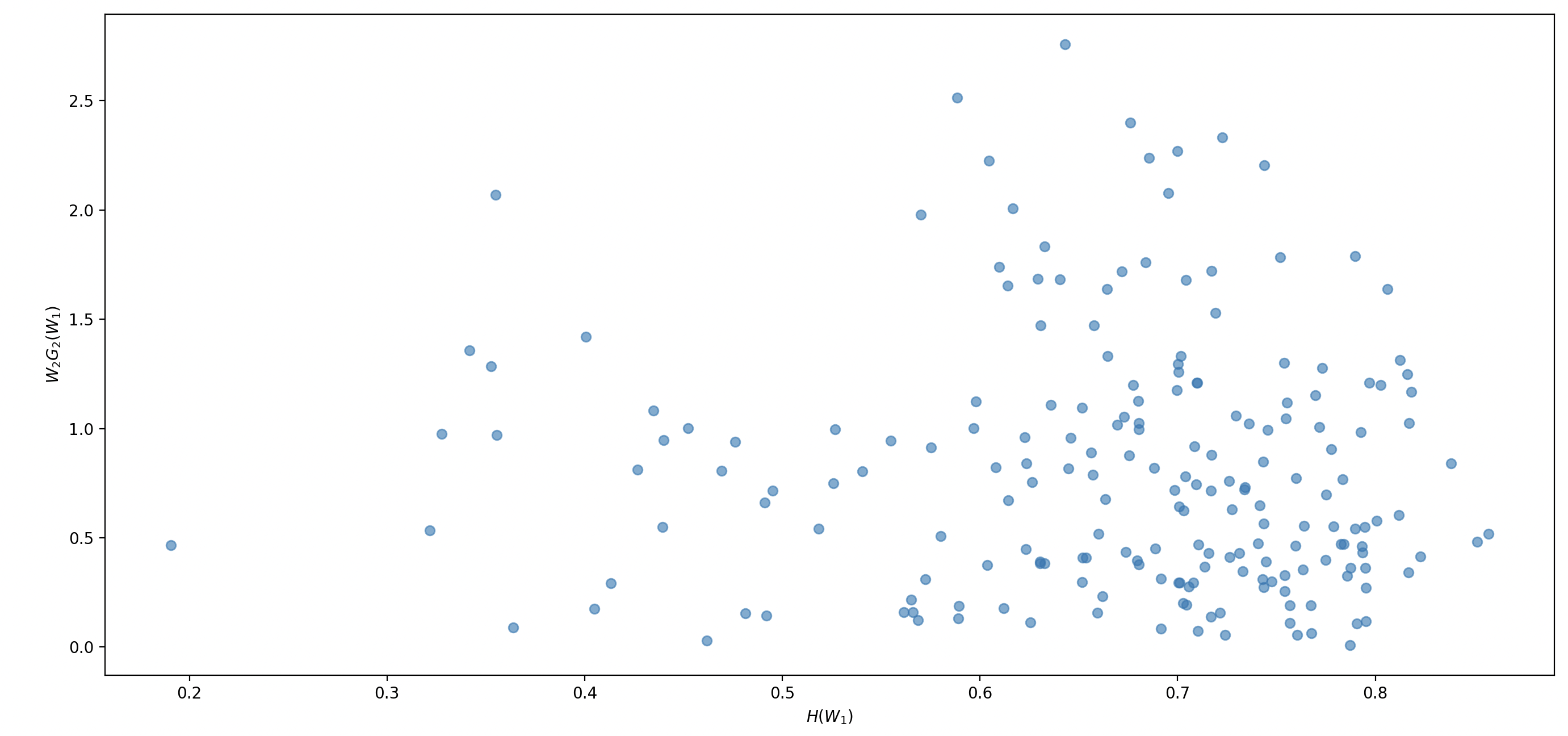}
    \caption{\small The joint distribution of $H(W_1)$ and $W_2G_2(W_1)$ in the case $d=2$ with $b_1=1.9$, $b_2=1.2$. }
    \label{fign}
\end{figure}

The second simulation study is related to \cite{LS}, where a computational analysis is performed to analyze an ultrasensitive sequencing dataset \cite{St} targeting five genomic locations using the SiMSen-seq protocol. Briefly, the protocol comprises a barcoding PCR step followed by a general amplification PCR step. The mathematical model for the amplification PCR step in \cite{LS} fits our definition of a MPCR branching process with $d=5$ and $v_1=\ldots=v_5=0.9$. For illustration purposes, we use the initial values
\begin{equation}\label{ini}
 Z_1(0)=16,\quad Z_2(0)=8,\quad Z_3(0)=4,\quad Z_4(0)=2,\quad Z_5(0)=1.
\end{equation}

To match $K=10^8$ in the study \cite{LS}, we set $K=(1.9)^{29}\approx 121{,}300{,}000$. Since the experimental data are based on $26$ rounds of amplification PCR, we verify Theorem~\ref{thm} in the form
\[
K^{-1}(Z_1(26),\ldots,Z_5(26))
\approx 
\bar W W^{-1}f^{(-3)}(H(W)),
\]
which involves the third iterate of the function
\[
f^{(-1)}(x)
=
\frac{x - 1.9 + \sqrt{x^2 + 0.2x + 3.61}}{2},
\qquad x \ge 0.
\]

Figure~\ref{fig3} demonstrates the good agreement provided by this approximation. The graph shows a clear separation between the components of 200 simulated vectors $\bar Z(26)$, 
reflecting the differences between the initial values \eqref{ini}.

\begin{figure}[htbp]
    \centering
    \includegraphics[height=7cm]{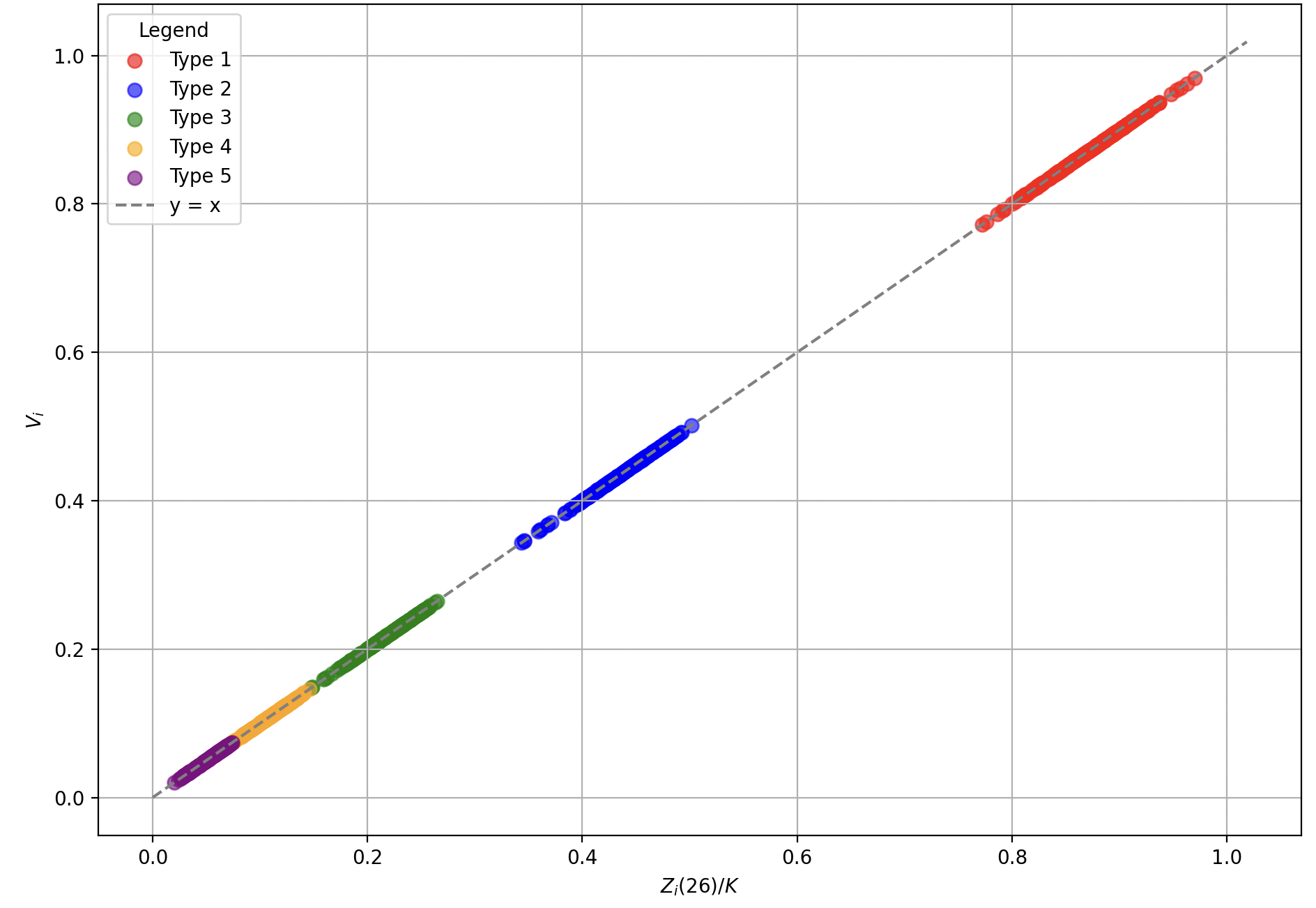}
    \caption{\small An illustration of Theorem~\ref{thm} in the case $d=5$, $b_1=b_2=b_3=b_4=b_5=1.9$, and $K=b_1^{-29}$.  Here the directly simulated values $Z_i(26)/K$ are matched against the limiting values $V_i=W_iW^{-1}f^{(-3)}(H(W))$.
}
    \label{fig3}
\end{figure}

\section{Proof of Theorem \ref{mainthm}}\label{Prf}

The proof implements the approximation method introduced in \cite{Fima17PCR} and utilizes some specific properties of the function $\bar F$ derived in Appendix \ref{app:A}. Throughout this paper, we use the $\ell^1$ norm for vectors and the corresponding operator norm for matrices. 
Generic constants, whose values may change from line to line, are denoted by $C$.
%
%

Fix some constant $c\in (\frac 1 2,1)$ and put $\kappa_c:=\lfloor c\kappa\rfloor$. To prove Theorem \ref{mainthm}, we will argue that $b_i^{-\kappa}  Z_i(\kappa) = (b_1/b_i)^\kappa   X_i(\kappa)$ can be approximated by  
$(b_1/b_i)^\kappa  F^{(\kappa-\kappa_c)}_i(\bar Y(\kappa_c)/K)$. 
More precisely, we will prove that, for all $i=1,\ldots,d$, 
\begin{align}
&(b_1/b_i)^\kappa F^{(\kappa-\kappa_c)}_i\big(\bar Y(\kappa_c)/K\big)  \xrightarrow{\P}  W_i G_i\big(W_0), \label{limi}\\
&(b_1/b_i)^\kappa\Big(X_i(\kappa) -   F^{(\kappa-\kappa_c)}_i\big(\bar Y(\kappa_c)/K\big)\Big) \xrightarrow{\P}0. \label{appi}
\end{align}

For a pair of integers $m\le n$, let us denote by  $\bar \Phi^{m,n}(\bar x)$  the solution to recursion \eqref{stochdyn} at time $n$ subject to initial condition 
$\bar X(m) =\bar x\in K^{-1} \mathbb N$ at time $m$. In this notation, 
$$\bar X(\kappa) = \bar \Phi^{\kappa_c,\kappa}(\bar X(\kappa_c)).$$
To prove \eqref{appi}, we will show that, for all $i=1,\ldots,d$, 
\begin{align}
&(b_1/b_i)^\kappa \E \Big| \Phi^{\kappa_c,\kappa}_i(\bar X(\kappa_c)) -  F^{(\kappa-\kappa_c)}_i(\bar X(\kappa_c))\Big|\to 0, \label{app:1a}\\
&(b_1/b_i)^{\kappa}\Big(  F^{(\kappa-\kappa_c)}_i(\bar X(\kappa_c))  -  F^{(\kappa-\kappa_c)}_i(\bar Y(\kappa_c)/K)\Big)\xrightarrow{\P} 0. \label{app:1b}
\end{align}

\subsection*{Proof of \eqref{limi}}
Let us first verify \eqref{limi} for the first $d_0$  components. In this case  
$$
(b_1/b_i)^\kappa F^{(\kappa-\kappa_c)}_i\big(\bar Y(\kappa_c)/K\big) =
F^{(\kappa-\kappa_c)}_i\left(\frac{b_1^{-\kappa_c}\bar Y(\kappa_c)}{b_1^{\kappa-\kappa_c}} \right), \quad i = 1,\ldots ,d_0,
$$
and, due to \eqref{branchlim},
$$
b_1^{-\kappa_c}\bar Y(\kappa_c)\xrightarrow{\P} \big(W_1,\ldots, W_{d_0}, 0,\ldots,0\big).
$$
Consequently, the claim follows from Lemma \ref{lem:3.1} and the extended continuous mapping theorem (see, e.g., \cite[Theorem 18.11]{VdV}):
\begin{align*}
\bar F^{(\kappa-\kappa_c)}\left(\frac{b_1^{-\kappa_c}\bar Y(\kappa_c)}{b_1^{\kappa-\kappa_c}} \right)
&\xrightarrow{\P}  \frac{H(W_0)}{W_0}\big(W_1,\ldots,W_{d_0}, 0,\ldots,0\big) \\
&=
\big(W_1G_1(W_0),\ldots,W_{d_0}G_{d_0}(W_0), 0,\ldots,0\big),
\end{align*}
where the last equality holds by virtue of identity \eqref{GH}, proved in Lemma \ref{lem:A5}.

To prove \eqref{limi} for $i=d_0+1,\ldots, d$, define $m_i(n)=(b_i/b_1)^{\kappa_c}b_1^{-n}$, where $n=\kappa-\kappa_c$, and
\[W_i(\kappa_c)=Y_i(\kappa_c)/b_i^{\kappa_c}.\]
By Lemma \ref{lem:3.2}, with the  matrix $M(n)$  defined therein, for $i>d_0$, we have
\begin{align*}
(b_1/b_i)^\kappa & F^{(\kappa-\kappa_c)}_i\big(\bar Y(\kappa_c)/K\big) = (b_1/b_i)^{\kappa_c}(b_1/b_i)^{\kappa-\kappa_c}
  F^{(\kappa-\kappa_c)}_i((b_1^{-\kappa_c}\bar Y(\kappa_c))/b_1^{\kappa-\kappa_c}) \\
&=
\frac{b_i^{\kappa_c-\kappa}}{m_i(\kappa-\kappa_c) }F^{(\kappa-\kappa_c)}_i\big(M(\kappa-\kappa_c) \bar W(\kappa_c)\big)
\xrightarrow{\P} W_i G_i(W_0),
\end{align*}
where the convergence holds by \eqref{branchlim} and \eqref{Filim}.

 

\subsection*{Proof of \eqref{app:1a}}

The sequence  $\bar X(n)=\bar \Phi^{\kappa_c,n}(\bar X(\kappa_c))$ satisfies  recursion \eqref{stochdyn} started from $\bar X(\kappa_c)$ at time $\kappa_c$.
On the other hand, the sequence  $\bar x(n)=\bar F^{(n-\kappa_c)}(\bar X(\kappa_c))$ satisfies the recursion
\[
\bar  x(n) = \bar  F(\bar  x(n-1)), \quad n > \kappa_c, 
\]
subject to the same initial condition $\bar x(\kappa_c)=\bar X(\kappa_c)$. Thus for any $m\in \mathbb N$,
\begin{align}\label{detrec}
\bar \Phi^{\kappa_c,\kappa_c+m} (\bar X(\kappa_c))&-\bar F^{(m)}(\bar X(\kappa_c))\\
&=\bar F(\bar \Phi^{\kappa_c,\kappa_c+m-1} (\bar X(\kappa_c)))-\bar F(\bar F^{(m-1)}(\bar X(\kappa_c)))+K^{-1/2}\bar \eps(m).\nonumber
\end{align}
Observe that the stochastic perturbation term in \eqref{stochdyn}, by definition \eqref{epsn}, satisfies 
\begin{equation}\label{epsbnd}
\E \big(|\eps_i(n)|\,\big|\F_{n-1}\big)\le \sqrt{\E \big(\eps_i(n)^2|\F_{n-1}\big)} =  
\sqrt{\frac{v_i X_i(n-1)}{1+X(n-1)}}\le \min \Big(1,   \sqrt{ X_i(n-1)}\Big).
\end{equation}

Since, by Lemma \ref{lem:Lip}, $\bar F$ is $b_1$-Lipschitz,  relations \eqref{detrec} and \eqref{epsbnd} imply
\begin{equation}\label{prev}
\E \left(\big\|\bar \Phi^{\kappa_c,\kappa_c+m} (\bar X(\kappa_c))-\bar F^{(m)}(\bar X(\kappa_c))\big\|\Big|\F_{\kappa_c}\right) \le 
K^{-1/2}\sum_{\ell=1}^{m} b_1^{m-\ell}\le C  b_1^{m-\kappa/2}.
\end{equation}
In particular, 
$$
\E \Big|\Phi^{\kappa_c,\kappa}_i(\bar X(\kappa_c))-F^{(\kappa-\kappa_c)}_i(\bar X(\kappa_c))\Big| \le 
C b_1^{\kappa/2-\kappa_c},
$$
yielding \eqref{app:1a} in the case $i \le d_0$, since $c \in \left(\tfrac{1}{2},1\right)$.

We now consider the case $i > d_0$. Using \eqref{detrec} and \eqref{ineq}, we obtain
\begin{multline*}
\big|\Phi_i^{\kappa_c,\kappa_c+m} (\bar X(\kappa_c)) - F^{(m)}_i(\bar X(\kappa_c))\big|   \le 
b_i \big|\Phi_i^{\kappa_c,\kappa_c+m-1}(\bar X(\kappa_c)) - F^{(m-1)}_i(\bar X(\kappa_c))\big| \\
+ F^{(m-1)}_i(\bar X(\kappa_c))\Big\|
\bar \Phi^{\kappa_c,\kappa_c+m-1}(\bar X(\kappa_c)) - \bar F^{(m-1)}(\bar X(\kappa_c))
\Big\| + K^{-1/2} |\eps_i(m)|.
\end{multline*}
In view of \eqref{epsbnd}, 
$$
K^{-1/2}\E |\eps_i(m)| 
\le 
K^{-1/2}\sqrt{\E  X_i(m-1)}
 \le K^{-1} \sqrt{ \E Y_i(m-1)} \le  b_i^{m/2} b_1^{-\kappa}.
$$
Therefore,  
\begin{align*}
(b_1/b_i)^\kappa&\E \big|\Phi_i^{\kappa_c,\kappa}(\bar X(\kappa_c)) - F^{(\kappa-\kappa_c)}_i(\bar X(\kappa_c))\big|  \\
&\le
b_1^\kappa \sum_{m=1}^{\kappa-\kappa_c} b_i^{-\kappa_c-m}
\E  \Big(F^{(m-1)}_i(\bar X(\kappa_c))\Big\|
\bar \Phi^{\kappa_c,\kappa_c+m-1}(\bar X(\kappa_c)) - \bar F^{(m-1)}(\bar X(\kappa_c))
\Big\|\Big)\\
& +  \sum_{m=1}^{\kappa-\kappa_c} b_i^{-\kappa_c-m/2}.
\end{align*}
By the Markov property of the processes and since $F^{(m-1)}_i(\bar X(\kappa_c))$ is $\F_{\kappa_c}$-measurable, 
\begin{align*}
\E &\Big(F^{(m-1)}_i(\bar X(\kappa_c))\Big\|
\bar\Phi^{\kappa_c,\kappa_c+m-1}(\bar X(\kappa_c)) - F^{m-1}(\bar X(\kappa_c))
\Big\| \Big)\\
&=
\E  \Big( F^{(m-1)}_i(\bar X(\kappa_c))\E \Big(\Big\|
\bar \Phi^{\kappa_c,\kappa_c+m-1}(\bar X(\kappa_c)) - F^{(m-1)}(\bar X(\kappa_c))
\Big\|\, \Big|\F_{\kappa_c}\Big) \Big) \\
&\stackrel{\eqref{prev} }{\le}
\E ( F^{(m-1)}_i(\bar X(\kappa_c))) C b_1^{m-\kappa/2}\stackrel{\eqref{Fibnd} }{\le} 
  C b_i^{m}\E (X_i(\kappa_c) ) b_1^{m-\kappa/2} \stackrel{\eqref{ZleY} }{\le}  
  C b_i^{m+\kappa_c} b_1^{m-3\kappa/2}.
\end{align*}
Consequently, 
\begin{align*}
(b_1/b_i)^\kappa & \E \Big| \Phi^{\kappa_c,\kappa}_i(\bar X(\kappa_c)) -  F^{(\kappa-\kappa_c)}_i(\bar X(\kappa_c))\Big|  \le
C b_1^{-\kappa/2} \sum_{m=1}^{\kappa-\kappa_c} b_1^{m}   +  C b_i^{-\kappa_c} 
\le C b_1^{\kappa/2-\kappa_c}     
+
C b_i^{-\kappa_c},
\end{align*}
leading to \eqref{app:1a} as $c\in (\frac 1 2,1)$.

\subsection*{Proof of \eqref{app:1b}}
Iterating the bound \eqref{ineq} yields
\begin{align*}
&
\big|F^{(\kappa-\kappa_c)}_i (K^{-1} \bar Y(\kappa_c))-F^{(\kappa-\kappa_c)}_i(\bar X(\kappa_c))\big|  \le
b_i^{\kappa-\kappa_c} \big|K^{-1} Y_i(\kappa_c)-X_i(\kappa_c)\big|\\
&+\sum_{m=0}^{\kappa-\kappa_c-1} b_i^{\kappa-\kappa_c-m} F^{(m)}_i(K^{-1} \bar Y(\kappa_c))
\big\|\bar F^{(m)}(K^{-1}\bar Y(\kappa_c))-\bar F^{(m)}(\bar X(\kappa_c))\big\|.
\end{align*}
By \eqref{Fibnd} and the $b_1$-Lipschitz continuity of $\bar F$,
\begin{multline*}
F^{(m)}_i (K^{-1} \bar Y(\kappa_c))
\big\|\bar F^{(m)}(K^{-1}\bar Y(\kappa_c))-\bar F^{(m)}(\bar X(\kappa_c))\big\|\\
 \le  b_i^m K^{-1} Y_i(\kappa_c) b_1^m \big\| K^{-1}\bar Y(\kappa_c)-\bar X(\kappa_c)\big\|   \le
  b_i^mb_1^{m-2\kappa } Y_i(\kappa_c)\big\|  \bar Y(\kappa_c)-\bar Z(\kappa_c)\big\|.    
\end{multline*}
Thus we obtain 
\begin{multline*}
(b_1/b_i)^{ \kappa} \big| F^{(\kappa-\kappa_c)}_i(K^{-1}\bar Y(\kappa_c))- F^{\kappa-\kappa_c}_i(\bar X(\kappa_c))\big| \\
    \le  b_i^{ -\kappa_c}\big| Y_i(\kappa_c)-Z_i(\kappa_c)\big|
+  Cb_i^{ -\kappa_c}  Y_i(\kappa_c)   b_1^{ -\kappa_c}  \big\| \bar Y(\kappa_c)-\bar Z(\kappa_c)\big\|,
\end{multline*}
and since $Y_i(\kappa_c) = O_\P(b_i^{\kappa_c} )$, it remains to apply the next lemma.

\begin{lem}\label{last}
Under the conditions of Theorem \ref{mainthm}, for any $c\in(1/2,1)$,
\[
b_i^{ -\kappa_c}\E\big| Y_i(\kappa_c)-Z_i(\kappa_c)\big| \to0, \quad i=1,\ldots,d.
\]
\end{lem}
 
\begin{proof}
Fix a constant $\gamma \in (c,1)$ and let $V(n)$ be a vector of independent Galton-Watson processes 
$$
V_i(n)=V_i(n-1) + \sum_{j=1}^{V_i(n-1)} \zeta_{i,j,n}, \quad i=1,\ldots,d,
$$
subject to $V_i(0)=Z_i(0)$, where, for the same random variables $U_{i,j,n}$ as in \eqref{xieta},
$$
\zeta_{i,j,n} =\mathbf{1}\left\{U_{i,j,n} \leq \frac{v_i}{1+dK^{\gamma-1}} \right\}.
$$
By this construction, for any $i=1,\ldots,d$, 
\begin{align*}
& V_i(n)\le Y_i(n), \quad \forall n\in \mathbb N,
\\
&
V_i(n) \le Z_i(n),  \quad \forall n\le\tau, 
\end{align*}
where $\tau$ is the stopping time at which some component of $\bar Z_n$ exceeds $K^\gamma$:
$$
\tau =\min \Big\{m\ge 0: \max_{i=1,\ldots,d}Z_i(m)\ge K^\gamma\Big\}.
$$
Then for any $n\in \mathbb N$,
\begin{equation}\label{dbnd}
0\le  Y_i(n) - Z_i(n) = Y_i(n) - V_i(n) + V_i(n)- Z_i(n) \le 
Y_i(n) - V_i(n) +  V_i(n) \one{n>\tau}.
\end{equation}
Since $Y_i(n)$ and $V_i(n)$ are Galton-Watson processes 
\begin{align*}
0\le \E (Y_i(\kappa_c)- V_i(\kappa_c) )&=  Z_i(0) b_i^{\kappa_c}-Z_i(0)\left(1+\frac{v_i}{1+dK^{\gamma-1}}\right)^{\kappa_c}  \\
&=
Z_i(0)b_i^{\kappa_c}-Z_i(0)b_i^{\kappa_c}\left(1-\frac{ (1-b_i^{-1})dK^{\gamma-1}}{1+dK^{\gamma-1}}\right)^{\kappa_c} \le  C \kappa_c b_i^{\kappa_c} K^{\gamma-1}.
\end{align*}
By the Cauchy-Schwarz inequality and the formula for the second moment of the Galton-Watson process \cite{Haccou}, 
$$
\E (V_i(\kappa_c) \one{\kappa_c>\tau}) \le \sqrt{\E ( V_i(\kappa_c)^2) \P(\kappa_c>\tau)} \le  C b_i^{\kappa_c} \sqrt{\P(\kappa_c>\tau)}.
$$
 Observe also that   
\begin{align*}
\P(\kappa_c>\tau) &\le  \P\left(\max_{1\le i\le d}\max_{m\le \kappa_c}Z_i(m)\ge K^\gamma\right) \le \sum_{i=1}^d \P\left(\max_{m\le \kappa_c}Z_i(m)\ge K^\gamma\right) \\
&\le 
 \sum_{i=1}^d \P\left(\max_{m\le \kappa_c}Y_i(m)\ge K^\gamma\right) \le 
 \sum_{i=1}^d \P\left(\max_{m\le \kappa_c} b_i^{-m}Y_i(m)\ge  (b_1/b_i)^{\kappa_c} K^{\gamma-c}\right)  \\
 &\stackrel{\dagger} \le\sum_{i=1}^d  Z_i(0) (b_i/b_1)^{\kappa_c} K^{c-\gamma} \le   C K^{c-\gamma},
\end{align*}
where $\dagger$ holds by Doob's inequality. 
Plugging these estimates into \eqref{dbnd} we obtain 
$$
0\le b_i^{-\kappa_c}\E\big(Y_i(\kappa_c) - Z_i(\kappa_c)\big) \le C\kappa_c  K^{\gamma-1} +    C K^{(c-\gamma)/2},
$$
which leads to the statement of Lemma \ref{last}.
\end{proof}

\section{Proof of Theorem \ref{thm}}
Due to  \eqref{claim2} and the identity \eqref{GH},
$$
\bar X(\kappa) \to \bar X(\infty)= \Big(\frac{W_1}{W_0},\ldots, \frac{W_{d_0}}{W_0},0,\ldots,0\Big)H(W_0).
$$
Since $\bar F$ is continuous and it has a continuous inverse (Lemma \ref{lem:inv}), it follows from \eqref{stochdyn} that for any fixed $n\in \mathbb Z$, 
$$
\bar X(\kappa +n) \xrightarrow{\P}  \bar F^{(n)}\big(\bar X(\infty)\big)=\Big(\frac{W_1}{W_0},\ldots, \frac{W_{d_0}}{W_0},0,\ldots,0\Big)f^{(n)}(H(W_0)),
$$
where the last equality follows from Lemma \ref{lem:iter}.

\appendix 
\section{Properties of $\bar F$} \label{app:A}

In this section we summarize some properties of the function $\bar F$ defined in \eqref{barF}, instrumental in the proofs of the main theorems.

\subsection{Invertibility.}

\begin{lem}\label{lem:inv}
The function $\bar F$ has a continuous inverse on $\Real_+^d$. The subset 
\begin{equation}\label{Gamma}
\Gamma = \big\{\bar x\in \Real_+^d: x_{d_0+1}=...=x_d=0\big\}
\end{equation}
is both backward and forward invariant with respect to $\bar F$. The restriction of the inverse to $\Gamma$ admits the formula:
\begin{equation}\label{invinv}
\bar F^{-1}(\bar y) = \frac {f^{-1}(y)}{y}\bar y, \quad \bar y\in \Gamma.
\end{equation}
\end{lem}

\begin{proof}
Fix a vector $\bar y\in \Real_+^d$ and consider the equation $\bar F(\bar x)=\bar y$, i.e.,
$$
x_i\Big(1+\frac{v_i }{1+x}\Big)=y_i, \quad i=1,\ldots,d.
$$
By rearranging and summing over $i$ we get  
$$
x = \sum_{i=1}^d\Big(1+\frac{v_i }{1+x}\Big)^{-1} y_i = \sum_{i=1}^d  \frac{1+x}{b_i+x }  y_i.
$$
Define the function  
$$
\psi(t) = t-\sum_{i=1}^d  \frac{1+t}{b_i+t }  y_i, \quad t\in \Real_+.
$$
Since   
$$
\psi(0)= -\sum_{i=1}^d  y_i/b_i  <0,\quad \psi'(t) =1 -\sum_{i=1}^d  \frac{v_iy_i}{(b_i+t)^2 }  , \quad 
\psi''(t) =  2\sum_{i=1}^d  \frac{v_iy_i}{(b_i+t)^3 } >0, 
$$
and $\psi(t)\to \infty$ as $t\to\infty$, the function $\psi$ has a unique root $\tau(\bar y)\in (0,\infty)$, continuous in $\bar y$. Then the inverse of $\bar F$ is given by 
\begin{equation}\label{inv}
F^{(-1)}_i(\bar y) =  \frac{1+\tau(\bar y)}{b_i+\tau(\bar y) }  y_i, \quad i=1,\ldots,d.
\end{equation}

By definition \eqref{barF}, $\bar F(\Gamma)\subseteq \Gamma$ and thus $\Gamma$ is forward invariant. Due to \eqref{inv}, $F^{-1}_i(\bar y) =0$ for $i>d_0$
when $y\in \Gamma$ and hence $\bar F^{-1}(\Gamma)\subseteq \Gamma$, i.e., $\Gamma$ is also backward invariant.
Moreover, for $y\in \Gamma$, summing up the equations in the system $\bar F(\bar x)=\bar y$ yields
$$
y = x\frac{b_1+x}{1+x} = f(x),
$$
and consequently,  
$$
y_i = x_i \frac{b_1+x}{1+x} = \frac{x_i}x f(x), \quad i=1,...,d_0.
$$
The formula \eqref{invinv} is obtained by substitution $x = f^{-1}(y)$.
\end{proof}

\subsection{Regularity and growth}

\begin{lem}\label{lem:Lip}
The function $\bar F$, defined in \eqref{barF}, is $b_1$-Lipschitz with respect to the $\ell_1$-norm, and it
satisfies  
\begin{equation}\label{ineq}
\big|F_i(\bar x)-F_i(\bar y)\big| \le b_i |x_i-y_i| +   x_i \|\bar x-\bar y\|, \quad \bar x,\bar y\in \Real_+^d,\quad i=1,\ldots,d.
\end{equation}
The $n$-th iterate of $\bar F$ admits the bound 
\begin{equation}
\label{Fibnd}
F^{(n)}_i (\bar x) \le b_i^n x_i, \quad \bar x\in \Real_+^d,\quad i=1,\ldots,d.
\end{equation}

\end{lem}
\begin{proof}
The partial derivatives of $\bar F$ satisfy the bounds 
$$
\partial_i F_i(\bar x)  = 
1+\frac{ v_i(1+x-x_i ) }{(1+x)^2} \le 1+  \frac{ v_1     }{ 1+x},
$$
and, for $k\ne i$, 
$$
\left|\partial_k F_i(\bar x)\right| = 
\Big|- \frac{v_ix_i}{(1+x)^2}\Big|\le   \frac{ v_1x_i}{(1+x)^2}.
$$
Therefore, the $\ell_1$-operator norm of the Jacobian matrix admits the estimate 
\begin{align*}
\big\|\nabla \bar F(\bar x)\big\|_{1} &=
\max_{1\le k\le d}\sum_{i=1}^d \big|\partial_k F_i(\bar x)\big| \\
&\le
\max_{1\le k\le d}\Big(
1+  \frac{ v_1     }{ 1+x}
+\sum_{i\ne k}     \frac{ v_1x_i}{(1+x)^2}
\Big) \le 1+    
\frac{ v_1(1+2x)   }{(1+x)^2 }\le 1+v_1=b_1,
\end{align*}
and, consequently, $\bar F$ is $b_1$-Lipschitz.

To prove \eqref{ineq}, note that, for any $a,b,c,d\in \Real_+$,
\begin{align*}
\Big|\frac{a}{1+a+b}-\frac{c}{1+c+d}\Big| 
&\le  
\Big|\frac{a}{1+a+b}-\frac{c}{1+c+b}\Big|+\Big|\frac{c}{1+c+b}-\frac{c}{1+c+d}\Big| 
\le |a-c| + c|b-d|.
\end{align*}
Therefore, for $\bar x,\bar y\in \Real_+^d$, 
\begin{align*}
\Big|F_i(\bar x)-F_i(\bar y)\Big| &\le    \big|x_i-y_i\big|+v_i \Big|
\frac{ y_i}{1+y_i+\sum_{j\ne i} y_j}-\frac{x_i}{1+x_i + \sum_{j\ne i} x_j}
\Big|  \\
&\le
\big|x_i-y_i\big|+v_i \big|x_i-y_i\big|+ v_i x_i 
  \sum_{j\ne i} \big|x_j -  y_j \big| \le 
  b_i \big|x_i-y_i\big| +   x_i \|\bar x-\bar y\|.
\end{align*}
The growth estimate in \eqref{Fibnd} holds since $F_i(\bar x)\le (1+v_i)x_i=b_i x_i$.
\end{proof}

\subsection{Iterates}

The following lemma reveals that the iterates of $\bar F$ have a particularly simple form if started from a vector with zeros in the last $d-d_0$ entries. 
\begin{lem}\label{lem:iter}
For any $n\in \mathbb Z$, 
\begin{equation}\label{Ffn}
\bar F^{(n)} (\bar x)= \frac{f^{(n)} (x)\bar x}{x},\quad \bar x\in \Gamma,
\end{equation}
where $\Gamma$ is defined in \eqref{Gamma} and $f^{(n)}$ is the $n$-th iterate of the function in \eqref{fx}.
\end{lem}

\begin{proof}
By definitions \eqref{barF} and \eqref{fx} and assumption \eqref{viq}, for $\bar x\in \Gamma$,
$$
F_i(x) = \begin{cases}
\dfrac{x_i}x f(x), & i\le d_0 \\
0,& i>d_0,
\end{cases}
$$  
which, in vector notation, verifies \eqref{Ffn} for $n=1$. By Lemma \ref{lem:inv}, the subset $\Gamma$ is forward invariant, and hence \eqref{Ffn} holds 
for an arbitrary $n>1$ by induction:
$$
\bar F^{(n)}(\bar x) = \bar F \big(\bar F^{(n-1)}(\bar x)\big) = \frac{f(F^{(n-1)}(\bar x))}{F^{(n-1)}(\bar x)}\bar F^{(n-1)}(\bar x) =
\frac{f(f^{(n-1)}(x))}{f^{(n-1)}(x)}\frac{f^{(n-1)}(x)}{x}\bar x = \frac{f^{(n)}(x)}{x}\bar x.
$$
The same argument applies for $n<0$ by virtue of the formula \eqref{invinv}. 
\end{proof}

\subsection{Scaling limits}
In this section we derive two key limits for the iterates of $\bar F$ under a suitable scaling. 
The following lemma addresses the first $d_0$  fastest-growing components. 

\begin{lem}\label{lem:3.1}
Let $H$ be the limit in \eqref{Hx}. Then as $n\to \infty$,
\begin{equation}\label{F1lim}
\bar F^{(n)} (\bar x/b_1^n)\to \frac  {H(x_0)}{x_0 } (x_1,\ldots,x_{d_0},0,\ldots,0),
\quad  \bar x\in \Real_+^d,  
\end{equation}
where the convergence is uniform on compact subsets of $\Real_+^d$ and $x_0=x_1+\ldots +x_{d_0}$.
\end{lem}

\begin{proof} 
Let $\bar y \in \Real_+^d$ be the vector with $y_i = x_i$ for $1\le i\le d_0$ and $y_i=0$ for $i>d_0$. Due to \eqref{Ffn},
$$
\bar F^{(n)} (\bar y/b_1^n)= \frac{f^{(n)} (y/b_1^n)}{y}\bar y \to \frac{H(y)}{y}\bar y = \frac  {H(x_0)}{x_0} (x_1,\ldots,x_{d_0},0,\ldots,0),\quad n\to\infty,
$$
where the convergence is uniform on compacts as shown in \cite[Section 3]{Fima17PCR}. It remains to check that 
\begin{equation}\label{S1}
 \sup_{0\le x\le R}\big\|\bar F^{(n)} (\bar y/b_1^n)-\bar F^{(n)} (\bar x/b_1^n)\big\|\to 0,\quad n\to\infty,\quad \forall R>0.
\end{equation}

Since $F^{(n)} _i(\bar y)=0$ for $i>d_0$ and all $n$, we have   
$$
\big\|\bar F^{(n)} (\bar y/b_1^n)-\bar F^{(n)} (\bar x/b_1^n)\big\|=  \delta_n(\bar x/b_1^n) +\sum_{i=d_0+1}^d  F^n_i(\bar x/b_1^n),
$$
where
$$\delta_n(\bar x):= \sum_{i=1}^{d_0} \big|  F^{(n)}_i(\bar y)-  F^{(n)}_i(\bar x)\big|.$$
Thus in view of \eqref{Fibnd}, to establish \eqref{S1}, it suffices to argue that 
\begin{equation}\label{dnn}
\sup_{0\le x\le R}\delta_n(\bar x/b_1^n) \to 0,\quad n\to\infty.
\end{equation}

To prove \eqref{dnn}, observe that summing the inequalities in \eqref{ineq} over $i\in \{1,\ldots,d_0\}$ yields
\begin{align*}
\delta_{k}(\bar x)  & \le\, b_1 \delta_{k-1 }(\bar x) + F^{(k-1)}(\bar y)\Big(\delta_{k-1 }(\bar x)+\sum_{i=d_0+1}^d F_i^{(k-1)}(\bar x)\Big)  \\
&\le
b_1 \delta_{k-1 }(\bar x) + b_1^k x\Big(\delta_{k-1 }(\bar x)+x \sum_{i=d_0+1}^d b_i^k  \Big),
\end{align*}
where the last inequality holds by \eqref{Fibnd}. Since $\delta_0(\bar x)=0$, by iterating this ineqiality $m$ times we obtain
$$
\delta_m(\bar x )\le  x^2  \sum_{k=1}^m \big(b_1+b_1^k x\big)^{m-k}b_1^k  \sum_{i=d_0+1}^d b_i^k ,
$$
and, consequently, for $m\le n$,
\begin{equation}\label{delbnd}
\sup_{0\le x\le R }\delta_m(\bar x/b_1^n ) \le 
x^2 \sum_{k=1}^m \big(b_1+b_1^{k-n} x\big)^{m-k}  b_1^{k-2n} \sum_{i=d_0+1}^d b_i^k \le
 C\phi(R) \ b_1^{m -2n}      b_{d_0+1}^m,
\end{equation}
where
\[\phi(r)=r^2\sup_{ j\ge 0}\big(1+b_1^{-j}r\big)^{j}\]
is finite for any $r>0$.
Setting $m=n$ in \eqref{delbnd}  yields \eqref{dnn}.   
\end{proof}

\begin{lem}\label{lem:3.20}
The infinite product in \eqref{Gilim} converges  and the limit $G_i:\Real_+\mapsto \Real_+$ is a continuous, positive and strictly decreasing function.
\end{lem}
\begin{proof}
By the upper bound in \eqref{Hr},
$$
\sum_{m=1}^\infty \Big|\log \frac{1+\frac 1 {b_i}H\Big(\frac r{b_1^m}\Big) }{1+H\Big(\frac r{b_1^m}\Big) }\Big| \le 
\sum_{m=1}^\infty  H\Big(\frac r{b_1^m} \Big) \le r \frac {b_1}{b_1-1},\quad r\ge0,
$$
and therefore the infinite product in \eqref{Gilim} is absolutely convergent and defines a continuous positive function. 
Recall that  $H$ is an increasing function.  Since $b_i^{-1}<1$, the function 
$
r\mapsto (1+b_i^{-1}r)/(1+r)
$ 
is a decreasing function on $\Real_+$ and so is each term in \eqref{Gilim}. Consequently, $G_i$ decreases as well. 
\end{proof}

The next lemma establishes the relevant scaling limits for the last $d-d_0$ components.

\begin{lem}\label{lem:3.2}
 Define the diagonal matrix 
$$
M(n) :=\mathrm{diag}\big(b_1^{-n},\ldots, b_1^{-n}, m_{d_0+1}(n), ...,m_d(n)\big),
$$
and assume that
\[0<m_i(n)\le Cb_1^{-n},\quad i=d_0+1,\ldots, d.\]
Then, for any $i>d_0$ and $\bar x\in \Real_+^d$, 
\begin{equation}\label{Filim}
\frac 1{m_i(n)b_i^n} \, F^{(n)}_i(M(n) \bar x)\to  x_iG_i(x_0), \quad n\to \infty,
\end{equation}
where convergence is uniform on compact subsets of $\Real_+^d$.

\end{lem}
 
\begin{proof}

First, notice that for $m\ge1$,
\begin{align*}
F_0^{(m)}(\bar x)= & F_0^{(m-1)}(\bar x) \Big(1+ \frac{v_1}{1+ F^{(m-1)}(\bar x)}\Big)\le F_0^{(m-1)}(\bar x) \Big(1+ \frac{v_1}{1+ F_0^{(m-1)}(\bar x)}\Big)\\
&= f\big(F_0^{(m-1)}(\bar x) \big)\le f^{(m)}(x_0),
\end{align*}
since $F_0^{(0)}(\bar x)= x_0$ and $f$ is increasing.

To prove \eqref{Filim} it suffices to check that
\begin{equation}\label{Bounds}
\prod_{m=0}^{n-1}\frac1 {1+F^{(m)}_{d+1}(\bar x)}\le \frac{F^{(n)}_i(\bar x) }{x_i \prod_{m=0}^{n-1} \left(1+\frac {v_i}{1+  f^{(m)}(x_0)}\right)}\le \prod_{m=0}^{n-1} \Big(1+ f^{(m)}(x_0)-F^{(m)}_0(\bar x)  \Big),
\end{equation}
where
\[F^{(m)}_{d+1}(\bar x)=\sum_{j=d_0+1} ^dF^{(m)}_j(\bar x)\]
and  
\begin{align}
\label{lim-i}
&
b_i^{-n}\prod_{m=0}^{n-1} \Big(1+\frac {v_i}{1+f^{(m)}(x_0/b_1^n)}\Big)\to G_i(x_0),
\\
\label{lim-ii}
&
\prod_{m=0}^{n-1} \Big(1+ f^{(m)}(x_0/b_1^n)-F_0^{(m)}(M_n \bar x) \Big)
\to 1,
\\
&
\label{lim-iii}
\prod_{m=0}^{n-1}\Big(1+F^{(m)}_{d+1}(M_n \bar x)\Big) \to 1,
\end{align}
as $n\to\infty$, uniformly on compacts.

\subsubsection*{Proof of \eqref{Bounds}}
By definition \eqref{barF}, for $i> d_0$,
\[
\begin{aligned}
F^{(n)}_i(\bar x) &=\,  x_i \prod_{m=0}^{n-1} \left(1+\frac {v_i}{1+F^{(m)}(\bar x)}\right)
\le x_i \prod_{m=0}^{n-1} \left(1+\frac {v_i}{1+ F^{(m)}_0(\bar x)}\right) \\
&\le 
x_i \prod_{m=0}^{n-1} \left(1+\frac {v_i}{1+  f^{(m)}(x_0)}\right)
\prod_{m=0}^{n-1} \Big(1+ f^{(m)}(x_0)-F^{(m)}_0(\bar x)  \Big)
\end{aligned}
\]
giving the upper bound in \eqref{Bounds}.
The lower bound is obtained similarly: 
\[
\begin{aligned}
F^{(n)}_i(\bar x) \ge\,  & x_i \prod_{m=0}^{n-1} \Big(1+\frac {v_i}{1+f^{(m)} (x_0)+F^{(m)}_{d+1}(\bar x)}\Big)  \\
&\ge
x_i \prod_{m=0}^{n-1} \Big(1+\frac {v_i}{1+f^{(m)}(x_0)}\Big)\prod_{m=0}^{n-1}\frac1 {1+F^{(m)}_{d+1}(\bar x)}.
\end{aligned}
\]

\subsubsection*{Proof of \eqref{lim-i}}

Write the left hand side of \eqref{lim-i} as
\begin{equation}
\label{showme}
\prod_{m=0}^{n-1} \frac 1{b_i}\Big(1+\frac {v_i}{1+f^{(m)}(x_0/b_1^{n})}\Big)= A_n(x_0)
\prod_{m=0}^{n-1} \frac 1 {b_i}\Big(1+\frac {v_i}{1+H(x_0/b_1^{n-m})}\Big),
\end{equation}
where 
\begin{equation}\label{Abnd}
\Big|A_n(x_0) - 1\Big| \le C\sum_{m=0}^{n-1} \Big|f^{(m)}(x_0/b_1^{n})-H(x_0/b_1^{n-m})\Big|.
\end{equation}
For any $k\in \mathbb N$ and $r\ge0$,  
\begin{equation}\label{Happ}
\begin{aligned}
| H(r)-f^{(k)}(r/b_1^k) | \le  & \sum_{j=k}^\infty |f^{(j+1)}(r/b_1^{j+1})-f^{(j)} (r/b_1^j)|   \\
&\stackrel \dagger \le
\sum_{j=k}^\infty b_1^j |f(r/b_1^{j+1})- r/b_1^j| =  \sum_{j=k}^\infty \frac{r^2v_1}{b_1^{j+2}+rb_1}\le  Cr^2 b_1^{-k},
\end{aligned}
\end{equation}
where $\dagger$ holds since $f$ is $b_1$-Lipschitz.
Substituting this estimate into \eqref{Abnd} yields 
$$
\sup_{\|x\|\le R}\Big|A_n(x_0) - 1\Big| \le  C R^2 b_1^{-n}\sum_{m=0}^{n-1}  b_1^{m-n}  \to 0, \quad \forall R>0,
$$
as $n\to\infty$. The limit in \eqref{lim-i} now follows from \eqref{showme}, since
$$
\prod_{m=0}^{n-1} \frac 1 {b_i}\Big(1+\frac {v_i}{1+H(x_0/b_1^{n-m})}\Big) =
\prod_{k=1}^{n}   \frac {1+b_i^{-1}H(x_0/b_1^{k})}{1+H(x_0/b_1^{k})}\to G_i(x_0),
$$
 as $n\to\infty$, uniformly on compacts (see Lemma \ref{lem:3.20}).

\subsubsection*{Proof of \eqref{lim-ii} and  \eqref{lim-iii}}
Note that due to identity \eqref{Ffn} and the bound in \eqref{delbnd}
$$
\sup_{0\le x\le R}\left|f^{(m)}(x_0/b_1^n)-F^{(m)}_0(\bar x/b_1^n)\right| \le  \sup_{x\le R} \delta_{m}(\bar x/b_1^n) \le C\phi(R) \ b_1^{m -2n}      b_{d_0+1}^m.
$$
The convergence in \eqref{lim-ii} follows since 
$$
\sup_{x\le R} \sum_{m=0}^{n-1} \left|f^{(m)}(x_0/b_1^n)-F^{(m)}_0(\bar x/b_1^n)\right|
\le 
C\phi(R)
b_1^{-2n}\sum_{m=0}^{n-1}(b_{d_0+1}b_1)^m.
$$
The convergence in \eqref{lim-iii} holds due to \eqref{Fibnd} since for any $R>0$, 
$$
\sup_{\|x\|\le R}\sum_{m=0}^{n-1}   F^{(m)}_{d+1}(M_n \bar x) \le CR (b_{d_0+1}/b_1)^n.
$$

\end{proof}

The last lemma proves the identity in \eqref{GH}.

\begin{lem}\label{lem:A5}
$$
rG_1(r)=H(r), \quad r\ge0.
$$
\end{lem}

\begin{proof}
By definition \eqref{Hx} and continuity of $f$, 
$$
H(r)=\lim_{n\to\infty} f^{(n+1)}(b_1^{-n-1}r) = f\Big(\lim_{n\to\infty} f^{(n)}\big(b_1^{-n}(r/b_1)\big)\Big)= f(H(r/b_1)).
$$
Therefore, 
$$
H\big(r/b_1^{m-1}\big) = f(H(r/b_1^m)) = H(r/b_1^m)\frac{b_1+H(r/b_1^m)}{1+H(r/b_1^m)}.
$$
Substituting this into the definition of $G_1(x)$ we obtain: 
$$
G_1(r)=\lim_{n\to \infty}\prod_{m=1}^{n}\frac{1+\frac{1}{b_1}H(r/b_1^{m})}{1+H(r/b_1^{m})}=
\lim_{n\to \infty}\prod_{m=1}^{n}  \frac{H\big(r/b_1^{m-1}\big)}{b_1H(r/b_1^m)} = 
\lim_{n\to \infty} \frac{H(r)}{b_1^nH(r/b_1^n)}= \frac{H(r)}r,
$$
where the last equality holds since $H'(0)=1$, see \cite[Lemma 7]{BCK24}.
\end{proof}


\end{document}